\documentclass[a4paper,11pt]{amsart}
\usepackage{amsmath}
\usepackage{amssymb}
\usepackage{amsthm}
\usepackage[english]{babel}
\usepackage{hyperref}
\usepackage{mathrsfs}
\usepackage{tikz}
\usetikzlibrary{matrix,positioning,fit,calc}

\usepackage{enumitem}

\newtheorem{theor}{Theorem}[section]
\newtheorem{lemma}[theor]{Lemma}

\newtheorem{rem}[theor]{Remark}

\newcommand{\Ind}{\mathrm{Ind}}

\newcommand{\s}{{\sf S}}

\newcommand{\Md}{\!\mod}

\renewcommand{\epsilon}{\varepsilon}
\renewcommand{\phi}{\varphi}
\newcommand{\xymat}{\xymatrix@R=6pt@C=10pt}

\newcommand{\la}{\lambda}
\newcommand{\be}{\beta}

\newcommand{\de}{\delta}

\newcommand{\da}{{\downarrow}}
\newcommand{\ua}{{\uparrow}}

\def\RP{{\mathscr {RP}}}

\begin{document}

\title[Composition factors of 2-parts spin representations]{Composition factors of 2-parts spin representations of symmetric groups}

\author{\sc Lucia Morotti}
\address
{Institut f\"{u}r Algebra, Zahlentheorie und Diskrete Mathematik\\ Leibniz Universit\"{a}t Hannover\\ 30167 Hannover\\ Germany} 
\email{morotti@math.uni-hannover.de}


\begin{abstract}
Given an odd prime $p$, we identify composition factors of the reduction modulo $p$ of spin irreducible representations of the covering groups of symmetric groups indexed by partitions with 2 parts and find some decomposition numbers.
\end{abstract}

\maketitle


\section{Introduction}

It is well known that in characteristic 0 (pairs of) irreducible spin representations of the symmetric groups are labeled by partitions in distinct parts, that is strict partitions.  
Labelings for the spin irreducible modules in odd characteristic 
have been found 
in \cite{BK4,BK2}, based on ideas from \cite{lt} using Hecke algebras $A_{p-1}^{(2)}$ (the labelings of the simple modules in \cite{BK4,BK2} have been proved to be equivalent in \cite{ks}). If $p=0$ let $\RP_p(n)$ denote the set of strict partitions of $n$. If $p\geq 3$ let $\RP_p(n)$ denote the set of $p$-strict $p$-restricted partitions of $n$, that is partitions $\la\vdash n$ for which $\la_r$ is divisible by $p$ if $\la_r=\la_{r+1}$ and which satisfy $\la_r-\la_{r+1}\leq p-\de_{p\mid \la_r}$ for any $r\geq 0$. Further in either of the previous cases, for any partition $\la$ let $h_{p'}(\la)$ be the number of parts of $\la$ which are not divisible by $p$. Then
\begin{align*}
&\{D(\la,0)|\la\in\RP_p(n)\text{ and }n-h_{p'}\text{ is even}\}\\
&\cup\{D(\la,+),D(\la,-)|\la\in\RP_p(n)\text{ and }n-h_{p'}\text{ is odd}\}
\end{align*}
is a complete set of non-isomorphic irreducible spin representations of $\s_n$. For $\la\in\RP_p(n)$ we define the supermodule $D(\la)$ to be either $D(\la,0)$ or $D(\la,+)\oplus D(\la,-)$ (depending on the parity of $n-h_{p'}(\la)$). If $p=0$ we will write $S(\la,\ldots)$ for $D(\la,\ldots)$ and $S(\la)$ for $D(\la)$.

Not much is know about decomposition matrices of spin representations of symmetric groups.  Known results include the basic and second basic spin cases, see \cite{Wales}, some partial results for $p=3$ and $5$, see \cite{abo1,b,bmo} which use ideas from \cite{my}, leading modules appearing in $S(\la)$, see \cite{BK2,BK,BK3} and the weight 1 case, see \cite{Mu}. For blocks of weight 2 it is expected, due to \cite[Conjecture 6.2]{lt}, that results similar to those in \cite{f} hold.

In characteristic 2, which will not be considered here, no spin representation of $\s_n$ exists. Some results about decomposition matrices in this case, in particular about leading terms, can be found in \cite{ben,bo}.

In this paper we will study composition factors of the reduction modulo $p$ of the supermodules $S((\la_1,\la_2))$. For irreducible representations of symmetric groups an exact description of the composition factors of the reduction modulo $p$ of the modules $S^{(\la_1,\la_2)}$ had been obtained by James in \cite{j}. Although we cannot in general exactly compute the multiplicity of the composition factors of $S((\la_1,\la_2))$, we will describe the composition factors of $S((\la_1,\la_2))$ which are not composition factors of $S((\mu_1,\mu_2))$ with $\mu_1+\mu_2=\la_1+\la_2$ and $\mu_1>\la_1$ and compute their multiplicity in $S((\la_1,\la_2))$.

Before stating results about decomposition matrices, we need to define certain particular partitions. Let $n=bp+c$ with $0\leq c<p$. If $n=0$ define $\be_n:=()$. If $n>0$ define $\be_n:=(p^b,c)$ if $c>0$ or $\be_n:=(p^{b-1},p-1,1)$ if $c=0$. Then $\be_n\in\RP_p(n)$. It is known that $\be_n$ is the partition labeling the basic spin modules in characteristic $p$. Further let $\ell=\ell_p:=(p-1)/2$ and define partitions $\mu_k$ for $1\leq k\leq \ell$ as follows:

{\small
\[\begin{array}{|c|c|c|}
\hline
c&k&\mu_k\\
\hline\hline
0&1&(p^{b-2},p-1,p-2,2,1)\\
\hline
0&2\leq k\leq \ell&(p^{b-1},p-k,k)\\
\hline

1&1&(p^{b-1},p-2,2,1)\\
\hline
1&2\leq k\leq \ell&(p^{b-1},p+1-k,k)\\
\hline

2\leq c\leq p-2&1\leq k\leq \lceil c/2\rceil-1&(p^b,\lfloor c/2\rfloor+k,\lceil c/2\rceil-k)\\
\hline
2\leq c\leq p-2&\lceil c/2\rceil&(p^{b-1},p-1,c,1)\\
\hline
2\leq c\leq p-2&\lceil c/2\rceil+1\leq k\leq \ell&(p^{b-1},p+\lceil c/2\rceil-k,\lfloor c/2\rfloor+k)\\
\hline

p-1&1\leq k\leq \ell-1&(p^b,\ell+k,\ell-k)\\
\hline
p-1&\ell&(p^{b-1},p-1,p-2,2)\\
\hline
\end{array}\]
\vspace{1mm}
\centerline{\sc Table I}
}

\noindent
It can be checked that $\mu_k\in\RP_p(n)$ if $n\geq p\geq 5$ and $1+\de_{n=p}\leq k\leq \ell$.

If $n<p$ then $\tilde\s_n$ is semisimple (where $\tilde\s_n$ is a double cover of $\s_n$), so this case does not need to be considered. For $n\geq p$ the following results hold:

\begin{theor}\label{T090919}
Let $n=p\geq 3$. Define $D_0:=D(\be_p)=D((p-1,1))$ and for $1\leq j\leq \ell-1$ define $D_j:=D(\mu_{j+1})=D((p-j-1,j+1))$. Further define $D_{-1},D_\ell:=0$. Then $[S((p-j,j))]=[D_j]+[D_{j-1}]$ for $0\leq j\leq \ell$.
\end{theor}

\begin{theor}\label{T120919}
Let $n>p=3$ and $m:=\lfloor (n-1)/2\rfloor-1-\de_{n\equiv 3\Md 6}$. For $0\leq j\leq m$ define $D_j:=D(\be_{n-j}+\be_j)$. If $\la=(\la_1,\la_2)\in\RP_0(n)$ then any composition factor of the reduction modulo $3$ of $S(\la)$ is of the form $D_j$ with $0\leq j\leq\min\{\la_2,m\}$. Further if $\la_2\leq m$ then $[S(\la):D_{\la_2}]=2^a$ with $a=1$ if at least one of the following holds:
\begin{itemize}[nosep,leftmargin=11pt,labelwidth=0pt,align=left]
\item $\la_1>\la_2>0$ are both divisible by $3$,

\item $\la_1>\la_2>0$, one of them is divisible by $3$ and $n$ is odd,

\item $\la_2=0$ and $n$ is even and divisible by $3$,
\end{itemize}
or $a=0$ else.
\end{theor}

\begin{theor}\label{T230819}
Let $n>p\geq 5$, $m:=\lfloor (n-1)/2\rfloor-\de_{n\equiv p\Md 2p}$ and $0\leq c\leq p-1$ with $n\equiv c\Md p$. For $1\leq k\leq \ell$ let $\mu_k$ be as in Table I. Define
\[D_j:=\left\{\begin{array}{ll}
D(\be_{n-j}+\be_j),&0\leq j\leq m-\ell,\\
D(\mu_{m+1-j}),&m-\ell<j\leq m\text{ and }n+c\text{ is even},\\
D(\mu_{\ell-m+j}),&m-\ell<j\leq m\text{ and }n+c\text{ is odd}.
\end{array}\right.\]
If $\la=(\la_1,\la_2)\in\RP_0(n)$ then any composition factor of the reduction modulo $p$ of $S(\la)$ is of the form $D_j$ with $0\leq j\leq\min\{\la_2,m\}$. Further if $\la_2\leq m$ then $[S(\la):D_{\la_2}]=2^a$ with $a=1$ if at least one of the following holds:
\begin{itemize}[nosep,leftmargin=11pt,labelwidth=0pt,align=left]
\item $\la_1>\la_2>0$ are both divisible by $p$,

\item $\la_1>\la_2>0$, one of them is divisible by $p$ and $n$ is odd,

\item $\la_2=0$ and $n$ is even and divisible by $p$,

\item $n\equiv 0\Md 2p$ and $\la=(n/2+1,n/2-1)$,
\end{itemize}
or $a=0$ else.
\end{theor}

In particular by Theorems \ref{T090919}, \ref{T120919} and \ref{T230819} we have that part of the decomposition matrix of the supermodules of $\widetilde\s_n$ is given by
\[\begin{tikzpicture}

\matrix(m)[matrix of nodes,column sep=-2pt,nodes in empty cells]{
$ $&$D_0$&$ $&$\cdots$&$D_{\overline{m}-b}$&other modules\\
\hline
$S((n))$&$a_0$&&&&$0$\\
$ $&&&$ $&&$ $\\
$ $&&&&&\\
$S((n-\overline{m}+b,\overline{m}-b))$&&&&$a_{\overline{m}-b}$&$0$\\
$S((n-\overline{m}+b-1,\overline{m}-b+1))$&$*$&&$\cdots$&$*$&$0$\\
$ $&$ $&&&$ $&$ $\\
$ $&&&&&\\
$S((n-\overline{m},\overline{m}))$&$*$&&$\cdots$&$*$&$0$\\
};

\draw (m-3-1) node {$\vdots$};
\draw (m-3-4) node {$\ddots$};
\draw (m-3-6) node {$\vdots$};
\draw (m-7-1) node {$\vdots$};
\draw (m-7-2) node {$\vdots$};
\draw (m-7-5) node {$\vdots$};
\draw (m-7-6) node {$\vdots$};
\draw ({$(m-1-1)!.9!(m-1-2)$} |- m.north) -- ({$(m-1-1)!.9!(m-1-2)$} |- m.south);
\draw ({$(m-2-4)!.7!(m-2-5)$}|-{$(m-2-4)!.2!(m-3-4)$}) node {\LARGE $0$};
\draw ({$(m-4-2)!.4!(m-4-3)$}|-{$(m-4-2)!.6!(m-5-2)$}) node {\LARGE $*$};


\end{tikzpicture}\]
with $\overline{m}=\lfloor(n-1)/2\rfloor$, $0\leq b\leq 2$ and $a_j\in\{1,2\}$. Theorems \ref{T090919} and \ref{T120919} can be easily obtained from known results and will be proved in Sections \ref{st1} and \ref{st2}. Theorem \ref{T230819} will be proved in Section \ref{st3} after having studied certain projective modules in Section \ref{sp}.

\section{Basic lemmas}

For $\la\in\RP_0(n)$ let $\la^R\in\RP_p(n)$ be the regularization of $\la$ as defined in \cite[Section 2]{BK3}. In particular if $\la=(\la_1,\ldots,\la_h)$ with $\la_r-\la_{r+1}\geq p+\de_{p\mid\la_r}$ for $1\leq r<h$ then $\la^R=\be_{\la_1}+\ldots+\be_{\la_h}$, while if $\la\in\RP_p(n)$ then $\la^R=\la$. The next lemma can be obtained combining \cite[Theorem 10.8]{BK2}, \cite[Theorem 10.4]{BK} and \cite[Theorem 4.4]{BK3}. For $\la\in\RP_p(n)$ let $a_p(\la):=0$ if $n-h_{p'}(\la)$ is even or $a_p(\la):=1$ if $n-h_{p'}(\la)$ is odd.

\begin{lemma}\label{L54}
Let $p\geq 3$ and $\la\in\RP_0(n)$. Then
\[[S(\la):D(\la^R)]=2^{(h(\la)-h_{p'}(\la)+a_0(\la)-a_p(\la^R))/2}\]
and if $\nu\in\RP_p(n)$ and $D(\nu)$ is a composition factor of the reduction modulo $p$ of $S(\la)$ then $\nu\unlhd\la^R$.
\end{lemma}

Normal nodes play an important role when considering branching (see for example \cite[\S9-a]{BK4} for the definition of normal nodes). The next results will be used in the proof of Theorem \ref{T230819}.

\begin{lemma}\cite[Theorem 8.4.5]{ks}\label{ln}
Let $p\geq 3$ and $\la\in\RP_p(n)$. If $A$ is a normal node of $\la$ and $\la\setminus A\in\RP_p(n-1)$ then $D(\la\setminus A)$ is a composition factor of $D(\la)\da_{\widetilde\s_{n-1}}$.
\end{lemma}

\begin{lemma}\cite[Theorem 9.13]{BK4}\label{ln2}
Let $p\geq 3$ and $\la\in\RP_p(n)$. If $\la$ has a unique normal node $A$ and $A$ has residue $0$ then $D(\la)\da_{\tilde\s_{n-1}}\cong D(\la\setminus A)$.
\end{lemma}

\section{Projective modules}\label{sp}

In this section we will construct certain projective modules whose structure will play a major role in the proof of Theorem \ref{T230819}. This section uses ideas from \cite{my} and extends results from \cite{abo1,b,bmo} to characteristic $p\geq 5$ for two-parts partitions.

The content of a node $(r,s)$ is given by $\min\{c-1,p-c\}$, where $1\leq c\leq p$ and $r\equiv c\Md{p}$. So nodes on any row have residues
\[0,1,\ldots,\ell-1,\ell,\ell-1,\ldots,1,0,0,1,\ldots,\ell-1,\ell,\ell-1,\ldots,1,0,\ldots.\]
The content of a partition is the multiset of the contents of its nodes. It is known that two partitions have the same content if and only if they have the same $p$-bar core. Further two irreducible modules are contained in the same block if and only if they are labeled by partitions with the same $p$-bar core (unless possibly if they are labeled by a $p$-bar core, in which case the weight is 0). In particular the content of a block is well defined. For $0\leq i\leq \ell$ and a $\widetilde\s_n$-module $M$ contained in the block(s) with content $I$, let $\Ind_i M$ be the block component(s) of $M\ua^{\widetilde\s_{n+1}}$ corresponding to the block(s) with content $I\cup\{i\}$.

\begin{lemma}\label{L220819_2}
Let $p\geq 5$ and $m:=\lfloor (n-1)/2\rfloor-\de_{n\equiv p\Md 2p}$. For $0\leq j\leq m$ there exist projective modules $[P_j:S(\mu)]=0$ if $\mu\in\RP_0(n)$ with $\mu_1>n-j$ and with $[P_j:S((n-j,j))]=2^a$, where $a=1$ if at least one of the following holds:
\begin{itemize}[nosep,leftmargin=11pt,labelwidth=0pt,align=left]
\item $n-j>j>0$ are both divisible by $p$,

\item $n-j>j>0$, one of them is divisible by $p$ and $n$ is even,

\item $j=0$ and $n$ is odd and divisible by $p$,

\item $n\equiv 0\Md 2p$ and $(n-j,j)=(n/2+1,n/2-1)$,
\end{itemize}
or $a=0$ else.
\end{lemma}

\begin{proof}
For $I=(i_1,i_2,\ldots)$ let $P_I:=\ldots \Ind_{i_2}\Ind_{i_1}D_0(())$. Then $P_I$ is projective. For $0\leq j\leq m$ we will now construct a residue sequence $I_j$ and define a positive integer $k_j$. We will then show that there exist a projective modules $P_j$ with $[P_j]=[P_{I_j}]/k_j$ for which the lemma holds.

{\bf Case 1:} $j=0$. Let $I_0$ be the residues of $(n)$ (taken starting from the node $(1,1)$ until the node $(1,n)$) and $k_0:=2^{\lfloor (n-1)/2\rfloor-a}$.

{\bf Case 2:} $1\leq j\leq m$. Write $(n-j,j):=(bp+c,dp+e)$ with $2\leq c\leq p+1$ and $1\leq e\leq p$. The case $b=d$ and $(c,e)=(\ell+1,\ell)$ is excluded by assumption. Let $I_j$ be obtained as concatenation of the following tuples: $(0,1,0)$, then $d$ times the tuples $(2,3,\ldots,\ell)$, $(1,2,\ldots,\ell-1)$, $(\ell-1,\ell)$, $(\ell-2,\ell-3,\ldots,0)$, $(0)$, $(\ell-1,\ell-2,\ldots,1)$, $(1,0,0)$ and then $I'_j$ where $I'_j$ is the concatenation of the following:

{\bf Case 2.1:} $b>d$, $1\leq e\leq \ell$: $(2,3,\ldots,\ell)$, $(1,2,\ldots,e-1)$ followed by the residues of $(n-j)$, starting from that of the node $(1,dp+(p+3)/2)$. Further let $k_j:=2^{\lfloor (n-2)/2\rfloor+2d+\de_{e=\ell}-a}$.

{\bf Case 2.2:} $b>d$, $\ell+1\leq e\leq p-1$: $(2,3,\ldots,\ell)$, $(1,2,\ldots,\ell-1)$, $(\ell-1,\ell)$, $(\ell-2,\ell-3,\ldots,0)$, $(0)$, $(\ell-1,\ell-2,\ldots,p-e)$ followed by the residues of $(n-j)$, starting from that of the node $(1,(d+1)p+2)$. Further let $k_j:=2^{\lfloor (n-2)/2\rfloor+2d+1+\de_{e=p-1}-a}$.

{\bf Case 2.3:} $b>d$, $e=p$: $(2,3,\ldots,\ell)$, $(1,2,\ldots,\ell-1)$ ,$(\ell-1,\ell)$, $(\ell-2,\ell-3,\ldots,0)$, $(0)$, $(\ell-1,\ell-2,\ldots,1)$, $(1,0)$ followed by the residues of $(n-j)$, starting from that of the node $(1,(d+1)p+3)$. Further let $k_j:=2^{\lfloor (n-2)/2\rfloor+2d+2-a}$.

{\bf Case 2.4:} $b=d$, $1\leq e\leq\ell-1
$, $e<c\leq \ell+1$: $(2,3,\ldots,c-1)$, $(1,2,\ldots,e-1)$. Further let $k_j:=2^{\lfloor (n-2)/2\rfloor+2d}$.

{\bf Case 2.5:} $b=d$, $1\leq e\leq \ell$, $\ell+2\leq c\leq p+1$: $(2,3,\ldots,\ell)$, $(1,2,\ldots,e-1)$ followed by the residues of $(n-j)$, starting from that of the node $(1,dp+(p+3)/2)$. Further let $k_j:=2^{\lfloor (n-2)/2\rfloor+2d+\de_{e=\ell}-a}$.

{\bf Case 2.6:} $b=d$, $\ell+1\leq e\leq p-1$, $e<c\leq p$: $(2,3,\ldots,\ell)$, $(1,2,\ldots,\ell-1)$, $(\ell-1,\ell)$, $(\ell-2,\ell-3,\ldots,p-c)$, $(\ell-1,\ell-2,\ldots,p-e)$. Further let $k_j:=2^{\lfloor (n-2)/2\rfloor+2d+1-a}$.

{\bf Case 2.7:} $b=d$, $\ell+1\leq e\leq p-2
$, $c=p+1$: $(2,3,\ldots,\ell)$, $(1,2,\ldots,\ell-1)$, $(\ell-1,\ell)$, $(\ell-2,\ell-3,\ldots,0)$, $(0)$, $(\ell-1,\ell-2,\ldots,p-e)$. Further let $k_j:=2^{\lfloor (n-2)/2\rfloor+2d+1}$.

{\bf Case 2.8:} $b=d$, $p-1\leq e\leq p$, $c=p+1$: $(2,3,\ldots,\ell)$, $(1,2,\ldots,\ell-1)$, $(\ell-1,\ell)$, $(\ell-2,\ell-3,\ldots,0)$, $(\ell-1,\ell-2,\ldots,p-e)$, $(0)$. Further let $k_j:=2^{\lfloor (n-2)/2\rfloor+2d+1-a}$.

In any of the above cases, if $I_j=(i_1,i_2,\ldots)$ and $i_r=i_{r+1}$ for some $r$ then $i_r\in\{0,1,\ell-1\}$ and, if it exists, $i_{r+2}\not=i_r$. Let $x_j$ to be the number of successive tuples $(1,1)$ or $(\ell-1,\ell-1)$ which appear in $I_j$ and $y_j$ to be the number of single $0$. If $[P_{I_j}]=\sum_{\mu\in\RP_0(n)}c_{j,\mu}[S(\mu)]$  then $c_{j,\mu}$ is divisible by $2^{\lfloor (n-h(\mu))/2\rfloor+x_j+\max\{0,h(\mu)-y_j\}}$, since to obtain $\mu$ from $()$ adding a single node at each step we have to switch $\lfloor (n-h(\mu))/2\rfloor$ times from partitions with $a_0(\nu)=1$ to partitions with $a_0(\nu)=0$ and whenever adding two nodes of either the same residues $\not=0$ (so on different rows and far enough) or both of residue $0$, one on a new row and the other on a different row their order can be exchanged (and there exists at least $x_j+\max\{0,h(\mu)-y_j\}$ such pairs). It can then be checked that $[P_{I_j}:S(\mu)]$ is divisible by $k_j$ for each $\mu\in\RP_0(n)$. Further it can be computed that $[P_{I_j}:S((n-j,j))]=2^ak_j$ and that $[P_{I_j}:S(\mu)]=0$ if $\mu_1>n-j$.

The lemma then follows by taking $P_j$ with $[P_j]=[P_{I_j}]/k_j$.
\end{proof}

\begin{rem}
The sequences of residues $I_j$ given above roughly correspond to adding nodes according to the following sequence (as long as nodes are contained in the partition $(n-j,j)$ and possibly with minor modifications at the end) given by 
\[y_1,y_2,x_{1,1},\ldots,x_{1,8},x_{2,1},\ldots,x_{2,8},\ldots,\]
where the subsetes of nodes $y_i$ and $x_{1,i}$ are as follows
\[\begin{tikzpicture}

\matrix(m)[matrix of nodes,column sep=-2pt,nodes in empty cells]{
$1$&$2$&$3$&$\hspace{-8pt}\cdots\hspace{-3pt}$&$\ell-1$&$\ell$&$\ell+1$&$\ell+2$&$\ell+3$&$\hspace{-3pt}\cdots\hspace{-3pt}$&$p-2$&$p-1$&$p$&$p+1$&$p+2$\\
\hline
$y_1$&$y_1$&$x_{1,1}$&&&&$x_{1,1}$&$x_{1,3}$&$x_{1,5}$&&&&&$x_{1,5}$&$x_{1,7}$\\
$y_2$&$x_{1,2}$&&&$x_{1,2}$&$x_{1,4}$&$x_{1,4}$&$x_{1,6}$&&&$x_{1,6}$&$x_{1,8}$&&$x_{1,8}$\\
};

\draw  ({$(m-2-1)!.25!(m-2-2)$}|-{$(m-1-1)!.9!(m-2-1)$}) -- ({$(m-2-1)!.7!(m-2-2)$}|-{$(m-1-1)!.9!(m-2-1)$});
\draw  ({$(m-2-3)!.55!(m-2-4)$}|-{$(m-1-1)!.9!(m-2-1)$}) -- ({$(m-2-6)!.6!(m-2-7)$}|-{$(m-1-1)!.9!(m-2-1)$});
\draw  ({$(m-2-9)!.5!(m-2-10)$}|-{$(m-1-1)!.9!(m-2-1)$}) -- ({$(m-2-13)!.5!(m-2-14)$}|-{$(m-1-1)!.9!(m-2-1)$});
\draw  ({$(m-3-2)!.45!(m-3-3)$}|-{$(m-2-1)!.9!(m-3-1)$}) -- ({$(m-3-4)!.45!(m-3-5)$}|-{$(m-2-1)!.9!(m-3-1)$});
\draw  ({$(m-3-6)!.35!(m-3-7)$}|-{$(m-2-1)!.9!(m-3-1)$}) -- ({$(m-3-6)!.6!(m-3-7)$}|-{$(m-2-1)!.9!(m-3-1)$});
\draw  ({$(m-3-8)!.35!(m-3-9)$}|-{$(m-2-1)!.9!(m-3-1)$}) -- ({$(m-3-10)!.5!(m-3-11)$}|-{$(m-2-1)!.9!(m-3-1)$});
\draw  ({$(m-3-12)!.45!(m-3-13)$}|-{$(m-2-1)!.9!(m-3-1)$}) -- ({$(m-3-13)!.45!(m-3-14)$}|-{$(m-2-1)!.9!(m-3-1)$});


\end{tikzpicture}\]
and the subsets $x_{c,i}$ are obtained by shifting $x_{1,i}$ to the right by $(c-1)p$ columns.
\end{rem}

\section{Proof of Theorem \ref{T090919}}\label{st1}

From \cite[Theorem 4.4]{Mu} there exist simple supermodules $E_j$ for $0\leq j\leq \ell-1$ which are pairwise non-isomorphic such that, if $E_{-1}=E_\ell=0$, then $[S((p-j,j))]=[E_j]+[E_{j-1}]$ for $0\leq j\leq \ell$. The theorem then follows from Lemma \ref{L54}.

\section{Proof of Theorem \ref{T120919}}\label{st2}

Note that for $p=3$ we have that $\mu\in\RP_p(n)$ if and only if $\mu=\be_{\nu_1}+\ldots+\be_{\nu_h}$ with $(\nu_1,\ldots,\nu_h)$ a partition of $n$ with $\nu_r-\nu_{r+1}\geq3+\de_{3\mid\nu_r}$ for $1\leq r<h$. Further, if $(\pi_1,\ldots,\pi_k)$ is also a partition of $n$ with $\pi_r-\pi_{r+1}\geq 3+\de_{3\mid\pi_r}$ for $1\leq r<k$, then $\be_{\nu_1}+\ldots+\be_{\nu_h}\unlhd \be_{\pi_1}+\ldots+\be_{\pi_k}$ if and only if $(\nu_1,\ldots,\nu_h)\unrhd (\pi_1,\ldots,\pi_k)$.

The theorem then holds by Lemma \ref{L54}. See also \cite[Theorem 4.1]{bmo} for an alternative partial proof.

\section{Proof of Theorem \ref{T230819}}\label{st3}

We will first prove that any composition factor of $S(\la)$ with $\la\in\RP_0(n)$ with at most 2 rows is of the form $D_j$ for some $j$. This will be done by induction on $n$ (using Theorem \ref{T090919} if $n=p+1$). Let $\la\in\RP_0(n)$ with at most 2 rows and $D(\mu)$ be a composition factor of the reduction modulo $p$ of $S(\la)$. Then any composition factor of $D(\mu)\da_{\widetilde\s_{n-1}}$ is a composition factor of some $S(\nu)$ with $\nu\in\RP_0(n-1)$ with at most 2 rows. In particular by Lemma \ref{ln} there exists $\psi\in\RP_p(n-1)$ such that $D(\psi)$ is a composition factor of some $S(\nu)$ with $\nu\in\RP_0(n-1)$ with at most 2 rows such that $\mu$ is obtained from $\psi$ by adding an addable node. By induction we then have that $D(\mu)\cong D_j$ for some $0\leq j\leq m$ or $\mu$ is of one of the following forms:
\begin{itemize}[nosep,leftmargin=11pt,labelwidth=0pt,align=left]
\item $(p^a,b,c,1)$ with $a\geq 0$, $1<c<b<p-1$ and $(b,c)\not=(p-2,2)$,

\item $(p+1,p^a,b,c)$ with $a\geq 0$, $0<c<b<p$ and $(b,c)\not=(p-1,1)$,

\item $(p+1,p^a,p-1,b,1)$ with $a\geq 0$ and $1<b<p-1$,

\item $(p+1,p^a,p-2,2,1)$ with $a\geq 0$,

\item $(p+1,p^a,p-1,p-2,2)$ with $a\geq 0$,

\item $(p+1,p^a,p-1,p-2,2,1)$ with $a\geq 0$,

\item $(2p^a,b,p^c,d,1)$ with $a,c\geq 0$, $p<b<2p$ and $1<d<p-1$,

\item $(2p^a,2p-1,p+1,p^c,d,1)$ with $a,c\geq 0$ and $1<d<p-1$,

\item $(2p^a,b,p^c,p-1,2)$ with $a,c\geq 0$ and $p<b<2p$,

\item $(2p^a,2p-1,p+1,p^c,p-1,2)$ with $a,c\geq 0$,

\item $(2p^a,b,p+1,p^c,d)$ with $a,c\geq 0$, $p+1<b<2p-1$ and $1<d<p-1$,

\item $(2p^a,2p-1,p+2,p^c,d)$ with $a,c\geq 0$ and $1<d<p-1$,

\item $(2p^a,b,p+1,p^c,p-1,1)$ with $a,c\geq 0$ and $p+1<b<2p-1$,

\item $(2p^a,2p-1,p+2,p^c,p-1,1)$ with $a,c\geq 0$,

\item $(2p+1,2p^a,b,p^c,d)$ with $a,c\geq 0$, $p<b<2p$ and $0<d<p$,

\item $(2p+1,2p^a,2p-1,p+1,p^c,d)$ with $a,c\geq 0$ and $0<d<p$,

\item $(2p+1,2p^a,b,p^c,p-1,1)$ with $a,c\geq 0$ and $p<b<2p$,

\item $(2p+1,2p^a,2p-1,p+1,p^c,p-1,1)$ with $a,c\geq 0$.
\end{itemize}
It can be easily checked that in each of the above cases there exists a normal node $B$ with $\mu\setminus B\in\RP_p(n-1)$ and $D(\mu\setminus B)$ not a composition factor of some $S(\nu)$ with $\nu\in\RP_0(n-1)$ with at most 2 rows unless $\mu$ is of one of the following forms:
\begin{itemize}[nosep,leftmargin=11pt,labelwidth=0pt,align=left]
\item $(3,2,1)$ with $p\geq 7$,

\item $(p+1,2,1)$,

\item $(p+2,p+1,1)$,

\item $(2p+1,p+1,1)$.
\end{itemize}
Since the $p$-bar cores of these partitions have 3 rows, the corresponding modules are not a composition factors of some $S(\la)$ with $\la\in\RP_0(n)$ with at most 2 rows.

In view of Lemma \ref{L220819_2} the theorem holds, up to identification of the modules $D_j$ and their multiplicity in $S((n-j,j))$. For $0\leq j<m-\ell$ the theorem then holds by Lemma \ref{L54}, since in this case $(n-j,j)^R=\be_{n-j}+\be_j$. For $m-\ell<j\leq m$ we can then identify the modules $D_j$ comparing normal/removable nodes. We will now check the multiplicity. For $\mu\in\RP_p(n)$ with $a_p(\mu)=0$ let $P(\mu)=P(\mu,0)$ be the projective indecomposable  module with socle $D(\mu,0)$, while if $a_p(\mu)=1$ let $P(\mu)=P(\mu,+)\oplus P(\mu,-)$ be the sum of the projective indecomposable modules with socles $D(\mu,\pm)$. For $\mu\in\RP_p(n)$ and $\nu\in\RP_0(n)$ we have that:
\begin{itemize}[nosep,leftmargin=11pt,labelwidth=0pt,align=left]
\item if $a_p(\mu)=a_0(\nu)=0$ then
\begin{align*}
[S(\nu):D(\mu)]&=[S(\nu,0):D(\mu,0)],\\
[P(\mu):S(\nu)]&=[P(\mu,0):S(\nu,0)],
\end{align*}

\item if $a_p(\mu)=0$, $a_0(\nu)=1$ then
\begin{align*}
[S(\nu):D(\mu)]&=2[S(\nu,\pm):D(\mu,0)],\\
[P(\mu):S(\nu)]&=[P(\mu,0):S(\nu,\pm)],
\end{align*}

\item if $a_p(\mu)=1$, $a_0(\nu)=0$ then
\begin{align*}
[S(\nu):D(\mu)]&=[S(\nu,0):D(\mu,\pm)],\\
[P(\mu):S(\nu)]&=2[P(\mu,\pm):S(\nu,0)],
\end{align*}

\item if $a_p(\mu)=a_0(\nu)=1$ then for $\de=\pm$
\begin{align*}
[S(\nu):D(\mu)]&=[S(\nu,\de):D(\mu,+)]+[S(\nu,\de):D(\mu,-)]\\
&=[S(\nu,+):D(\mu,\de)]+[S(\nu,-):D(\mu,\de)],\\
[P(\mu):S(\nu)]&=[P(\mu,\de):S(\nu,+)]+[P(\mu,\de):S(\nu,-)]\\
&=[P(\mu,+):S(\nu,\de)]+[P(\mu,-):S(\nu,\de)].
\end{align*}
\end{itemize}

In particular $[S(\nu):D(\mu)]=2^{a_0(\nu)-a_p(\mu)}[P(\mu):S(\nu)]$. The multiplicity $[S((n-j,j)):D_j]$ is then as given in the theorem also for $m-(p-1)/2\leq j\leq m$, unless possibly if $n\equiv 0\Md 2p$ and $j=n/2-1$ or $n\equiv p\Md 2p$ and $j=(n-p)/2$. In either of these two cases $D_j=D((p^k,p-1,p-2,2,1))$ and $D_j\da_{\widetilde\s_{n-1}}=D((p^k,p-1,p-2,2))$ by Lemma \ref{ln2}.

If $n\equiv 0\Md 2p$ and $j=n/2-1$ then
\begin{align*}
[S((n-j,j))\da_{\widetilde\s_{n-1}}]&=[S((n/2,n/2-1))]+[S((n/2+1,n/2-2))].
\end{align*}
Further if $[D_r\da_{\widetilde\s_{n-1}}:D((p^k,p-1,p-2,2))]>0$ then
\begin{align*}
&[S((n-r,r))\da_{\widetilde\s_{n-1}}:D((p^k,p-1,p-2,2))]\\
&=c_r[S((n-r-1,n-r)):D((p^k,p-1,p-2,2))]\\
&\hspace{11pt}+c_{r-1}[S((n-r,r-1)):D((p^k,p-1,p-2,2))]\\
&>0
\end{align*}
and so $r=m=j$. Since $[S((n/2,n/2-1)):D((p^k,p-1,p-2,2))]=2$ and $[S((n/2+1,n/2-2)):D((p^k,p-1,p-2,2))]=0$ by induction, it follows that $[S((n-j,j)):D_j]=2$ in this case.

If $n\equiv p\Md 2p$ and $j=(n-p)/2$ then, since $n>p$,
\begin{align*}
[S((n-j,j))\da_{\widetilde\s_{n-1}}]=\,&2[S(((n+p)/2,(n-p)/2-1))]\\
&+2[S(((n+p)/2-1,(n-p)/2))].
\end{align*}
If $r\leq j$ and $[D_r\da_{\widetilde\s_{n-1}}:D((p^k,p-1,p-2,2))]>0$ then
\begin{align*}
&[S((n-r,r))\da_{\widetilde\s_{n-1}}:D((p^k,p-1,p-2,2))]\\
&=c_r[S((n-r-1,n-r)):D((p^k,p-1,p-2,2))]\\
&\hspace{11pt}+c_{r-1}[S((n-r,r-1)):D((p^k,p-1,p-2,2))]\\
&>0
\end{align*}
and so $r=j$. As $[S(((n+p)/2,(n-p)/2-1)):D((p^k,p-1,p-2,2))]=0$ and $[S(((n+p)/2-1,(n-p)/2)):D((p^k,p-1,p-2,2))]=1$, it follows that $[S((n-j,j)):D_{\la_2}]=2$ also in this case.

\section*{Acknowledgements}

The author was supported by the DFG grant MO 3377/1-1.

\end{document}